\documentclass[11pt,a4paper,reqno]{amsart}
\usepackage{color}
\usepackage{amsfonts,amsmath,amssymb,amsxtra,url,float} 
\allowdisplaybreaks[4] 
\usepackage[colorlinks,linkcolor=RoyalBlue,anchorcolor=Periwinkle,citecolor=Orange,urlcolor=Green]{hyperref} 
\usepackage[usenames,dvipsnames]{xcolor} 
\usepackage{enumitem}
\usepackage{geometry} 
\usepackage{rotating} 
\usepackage{lscape} 
\usepackage{graphicx} 
\usepackage{subfigure} 
\usepackage{tikz}
\usetikzlibrary{decorations.pathreplacing,decorations.markings}
 \tikzset{
  on each segment/.style={
    decorate,
    decoration={
      show path construction,
      moveto code={},
      lineto code={
        \path [#1]
        (\tikzinputsegmentfirst) -- (\tikzinputsegmentlast);
      },
      curveto code={
        \path [#1] (\tikzinputsegmentfirst)
        .. controls
        (\tikzinputsegmentsupporta) and (\tikzinputsegmentsupportb)
        ..
        (\tikzinputsegmentlast);
      },
      closepath code={
        \path [#1]
        (\tikzinputsegmentfirst) -- (\tikzinputsegmentlast);
      },
    },
  },
  mid arrow/.style={postaction={decorate,decoration={
        markings,
        mark=at position .5 with {\arrow[#1]{stealth}}
      }}},
}
\usetikzlibrary{arrows}
\usetikzlibrary{trees}

\usetikzlibrary{matrix}
\usetikzlibrary{patterns}
\usetikzlibrary{shadings} 
\usepackage[all]{xy} 
\usepackage{mathdots} 
\usepackage{yhmath} %
\usepackage{dsfont} 
\usepackage{cite}
\usepackage{mathrsfs} 
\usepackage{multicol} 
\numberwithin{figure}{section}

\usepackage{marginnote} 
\usepackage{graphicx} 
\usepackage{multicol} 

\usepackage{changepage} 
\usepackage{changes}  

\newtheorem{theorem}{Theorem}[section]
\newtheorem{lemma}[theorem]{Lemma}
\newtheorem{corollary}[theorem]{Corollary}
\newtheorem{main theorem}[theorem]{Main Theorem}
\newtheorem{proposition}[theorem]{Proposition}

\newtheorem{definition}[theorem]{Definition}

\newtheorem{remark}[theorem]{Remark}
\newtheorem{example}[theorem]{Example}
\newtheorem{notation}[theorem]{Notation}
\newtheorem{question}[theorem]{Question}

\usetikzlibrary{arrows}

\numberwithin{equation}{section}

\def\<{\langle} 
\def\>{\rangle} 
\def\NN{\mathbb{N}} 

\newcommand{\Pic}{F{\tiny{IGURE}}\ }
\newcommand{\modcat}{\mathsf{mod}}

\newcommand{\ind}{\mathsf{ind}}
\newcommand{\kk}{\mathds{k}} 
\newcommand{\Q}{\mathcal{Q}} 
\newcommand{\I}{\mathcal{I}} 
\newcommand{\J}{\mathcal{J}}

\newcommand{\End}{\mathrm{End}} %
\newcommand{\rad}{\mathrm{rad}} %
\newcommand{\op}{\mathrm{op}}
\def\A{\mathbb{A}}
\def\D{\mathbb{D}}
\def\dAA{\mathbb{A}\mkern-12.5mu\raisebox{0.33em}[]{${}^{\to}$}}
\def\dDD{\mathbb{D}\mkern-12.5mu\raisebox{0.33em}[]{${}^{\to}$}}

\def\dA{\overrightarrow{\pmb{A}}}
\def\dD{\overrightarrow{\pmb{D}}}

\def\rad{\mathrm{rad}}

\def\alg{\mathit{\Lambda}}
\def\source{\mathfrak{s}}
\def\target{\mathfrak{t}}
\def\Left{\mathrm{L}}
\def\Right{\mathrm{R}}
\def\Upper{\mathrm{U}}
\def\Down{\mathrm{D}}

\def\rmH{\mathrm{h}}
\def\rmV{\mathrm{v}}
\def\e{\varepsilon}
\def\frakp{\mathfrak{p}}
\def\frakS{\mathfrak{S}}
\def\frakI{\mathfrak{I}}

\newcommand{\idem}{\pmb{i}}
\newcommand{\z}{\pmb{z}}
\newcommand{\s}{\pmb{s}}
\newcommand{\itLamb}{\mathit{\Lambda}}

\def\spa{\mathrm{span}_{\kk}}
\def\nat{\pi}
\def\Aus{\mathrm{Aus}}

\newcommand{\Square}[2]{{_{\mathcal{S}}\overline{\mathcal{Q}}}_{#1,#2}}
\newcommand{\vertex}[2]{\mathfrak{v}_{#1,#2}}

\newcommand{\checkLiu}[1]{{\color{black}#1}}
\def\defines{\it\color{black!75}}

\newcommand{\gluing}{gluing\ }
\newcommand{\Gluing}{Gluing\ }
\newcommand{\g}{\mathfrak{g}}
\newcommand{\DS}{DS}
\newcommand{\PDS}{$1^{\mathrm{st}}$DS} 
\newcommand{\NPDS}{$2^{\mathrm{nd}}$DS} 
\newcommand{\GDS}{GDS}

\begin{document}

\title[The tensorial description of the Auslander algebras for of string algebras]{The tensorial description of the Auslander algebras of representation-finite string algebras}
\thanks{$^{\ast}$Corresponding author.}

\author{Hui Chen}
\address{School of Biomedical Engineering and Informatics, Nanjing Medical University,211166, Nanjing, Jiangsu, P. R. China.}
\email{huichen@njmu.edu.cn}

\author{Jian He$^\ast$}
\address{Department of Applied Mathematics, Lanzhou University of Technology, 730050, Lanzhou, Gansu, P. R. China}
\email{jianhe30@163.com}

\author{Yu-Zhe Liu}
\address{School of Mathematics and statistics, Guizhou University, 550025, Guiyang, Guizhou, P. R. China}
\email{liuyz@gzu.edu.cn / yzliu3@163.com}


\definecolor{section}{rgb}{0,0,0.5}

\begin{abstract}
The aim of this article is to study the Auslander algebra of any representation-finite string algebra. More precisely, we introduce the notion of gluing algebras and show that the Auslander algebra of a representation-finite string algebra is a quotient of a \gluing algebra of $\dA_n^e $. As applications, the Auslander algebras of two classes of string algebras whose quivers are Dynkin types $\A$ and $\D$ are described. Moreover, the representation types of the above Auslander algebras are also given exactly.

\end{abstract}

\maketitle


\section{Introduction}
The notion of string algebras, a class of special biserial algebra \cite{SW1983}, was introduced by Butler and Ringle in \cite{BR1987}.
Many important algebras are string, such as gentle algebras, Nakayama algebras.
There are many works on string algebras (for example, \cite[etc]{CG2017, CSZ2018, GKS2022, L2016, LPP2018, LZ2023, PP2016}). In particular, Butler and Ringel \cite{BR1987} classified the vertices of the Auslander-Reiten quiver of a string algebra into two classes,
namely the string modules and the band modules, as well as its arrows into four classes.
In \cite{CSZ2018}, Chen, Shen and Zhou constructed non-trivial indecomposable Gorenstein projective modules over string algebras.
Recently, Liu and Zhang \cite{LZ2023} characterized the Cohen-Macaulay Auslander algebra of any string algebra.

Let $A$ be an algebra of finite type and  $M_1,M_2,...,M_n$ be a complete set of representatives of the isomorphism classes of indecomposable $A$-modules. Then $A^{\Aus} = \End_A\left(\bigoplus^n_{i=1}M_i\right)$ is the Auslander algebra of $A$.  Auslander \cite{A1974} characterized
the algebras which arise this way as algebras of global dimension at most 2
and dominant dimension at least 2. Moreover, this construction of indecomposable $A$-modules induces a mutually inverse bijection between Morita
equivalence classes of representations-finite algebras and Morita equivalence classes of
Auslander algebras. This bijection is called the Auslander correspondence and the correspondence is generalized by Iyama to a higher dimensional version in \cite{I2007}. It is also known that Auslander algebras have close relations with quasi-hereditary algebras, preprojective algebras and projective quotient algebras, see \cite{CS2017, DR1989, GLS2007} for details.

The paper is devoted to study the Auslander algebra of any representation-finite string algebra. To be more precise, we propose a question as follows.
\begin{question}\rm\label{Intro}
How to describe the Auslander algebra of any representation-finite string algebra?
\end{question}

We recall the tensor product of rings, which is introduced by Frobenius and also plays an important role in the representation theory of algebras. As a special kind of tensor product of algebras, the enveloping algebra $A^e$ of $A$ is defined by $A^e = A\otimes_{\kk}A^{\op}$, where $A^{\op}$ is the opposite algebra of $A$ (see \cite{CE1956} for more details). It is worth nothing that any $(A,A)$-bimodule $M={_AM_A}$ is natural right $A^e$-module $M=M_{A^e}$ by
$$M_{A^e}\times A^e\longrightarrow M_{A^e},\ \ m\cdot(a\otimes a'):=a'ma. $$

In order to answer Question \ref{Intro}, we first introduce a class of finite dimensional algebra whose quiver $Q$ can be
``stacked'' by some diamond quivers.
We call the above class of algebras Generalized diamond-stacked (=GDS) algebras. Meanwhile, we also introduce the notion of gluing algebras (See Definition \ref{def:GDS} and \ref{Gluing algebras} for details on GDS and gluing algebras).
Let $\dA_n$ be the path algebra $\kk\dAA_n$ with $\dAA_n=$
\[1 \longrightarrow 2 \longrightarrow \cdots \longrightarrow n-1 \longrightarrow n\]
and $\dA_n^e = \dA_n\otimes_{\kk}\dA_n^{\op}$ its enveloping algebra. In this paper, we show the following result, which characterize the Auslander algebra of any representation-finite string algebra.

\begin{theorem} {\rm(Proposition \ref{prop:Aus-is-GDS} and Theorem \ref{thm:main})}\label{1}
The Auslander algebra $A^{\Aus}$ of arbitrary representation-finite string algebra $A$ is a GDS algebra. Moreover,  $A^{\Aus}$ is a quotient of a \gluing algebra of $\dA_n^e = \dA_n\otimes_{\kk}\dA_n^{\op}$.
\end{theorem}

As applications, we consider the Auslander algebras of $\kk$-algebras for type $\protect\dAA_n$ and type $\dDD_n$. Where $\dDD_n$ ($n\ge 4$) denote the Dynkin quiver of type $\D$ with linearly orientation
\[\xymatrix{  & 2 \ar[d]^{a_2} & & &\\
1 \ar[r]^{a_1}& 3\ar[r]^{a_3} & 4 \ar[r] & \cdots \ar[r]^{a_{n-1}} & n, }\]
and $\dD_n = \kk\dDD/\langle a_2a_3\rangle$.  By Theorem \ref{1}, we obtain the following result.

\begin{corollary} {\rm(Examples \ref{exp:A} and \ref{exp:D})}
$\dA_n$ and $\dD_n$ are string algebras, whose Auslander algebras are quotients of $\dA_n^e$.
\end{corollary}

This paper is organized as follows. In Section 2 we recall some preliminaries of
string algebras. In Section 3 we introduce the notions of GDS and gluing algebras and prove
some conclusions that we need to prove the main results. In Section 4, we prove Theorem \ref{1}.
In the last section, we give some applications.

 \section{Preliminaries}
\subsection{String algebras}

In this part, we recall some concepts for string algebras.
We always assume that $\kk$ is an algebraically closed field in our paper, and a quiver is a quadruple $\Q = (\Q_0, \Q_1, \source, \target)$,
where $\Q_0$ is the set of vertices, $\Q_1$ is the set of arrows, and $\source$ and $\target$ are functions $\Q_1 \to \Q_0$ sending any arrow $a\in \Q_1$ to its source and target, respectively.
Furthermore, we denote by $\Q_{\ell}$ the set of all paths of length $\ell$ (thus, naturally, $\Q_0$ is a set of all paths of length zero and $\Q_1$ is that of all paths of length one)
and write $\Q_{\ge t} = \bigcup\limits_{\ell=t}^{+\infty}\Q_{\ell}$.
If $a$ and $b$ are arrows such that $\target(a)=\source(b)$, the composition of $a$ and $b$ is denoted by $ab$.
All algebras in this paper are finite dimensional $\kk$-algebras, and, for any algebra $A$, all modules we consider are finitely generated right $A$-modules.

A {\defines monomial algebra} is a finite dimensional $\kk$-algebra which is Morita equivalent to $\kk\Q/\I$ where $\I$ is generated by some paths of length $\ge 2$. The bound quiver $(\Q, \I)$ of monomial algebra is called a {\defines monomial pair}.
Furthermore, we say a monomial pair $(\Q, \I)$ is a {\defines string pair} if it satisfies the following conditions:
\begin{itemize}
  \item[(S1)$_{\color{white}\Right}$]
     Any vertex of $\Q$ is the source of at most two arrows and the target of at most two arrows.
  \item[(S2)$_\Right$] For each arrow $\alpha:x\to y$, there is at most one arrow $\beta$
   whose source $\source(\beta)$ is $y$ such that $\alpha\beta\notin\I$.
  \item[(S2)$_\Left$] For each arrow $\alpha:x\to y$, there is at most one arrow $\gamma$
   whose target $\target(\gamma)$ is $x$ such that $\gamma\alpha\notin\I$.
\end{itemize}

\begin{definition}\rm
A finite-dimensional algebra $A$ is called a {\defines string algebra} if $A$ is Morita equivalent to $\kk\Q/\I$ where $(\Q, \I)$ is a string pair.
\end{definition}

\begin{example}\rm We provide some instances for string algebra.
\begin{itemize}
  \item[(1)]
  Gentle algebras are {string algebras}.
  Indeed, a finite-dimensional algebra $A$ is said to be {\defines gentle} if it is Morita equivalent to $\kk\Q/\I$ such that the following conditions hold.
  \begin{itemize}
    \item (S1), (S2)$_{\Right}$, (S2)$_{\Left}$.
    \item For each arrow $\alpha:x\to y$, there is at most one arrow $\beta$
       whose source $\source(\beta)$ is $y$ such that $\alpha\beta\in\I$.
    \item For each arrow $\alpha:x\to y$, there is at most one arrow $\gamma$
       whose target $\target(\gamma)$ is $x$ such that $\gamma\alpha\in\I$.
    \item All generators of $\I$ are paths of length two.
  \end{itemize}

  \item[(2)] Assume that $\dAA_n$ is the quiver of the form
    \[ 1 \longrightarrow 2 \longrightarrow \cdots \longrightarrow n-1 \longrightarrow n. \]
  For any admissible ideal $\I$, $\kk\dAA_n/\I$ is string.
  In particular, if $\I=0$, then $\kk\dAA_n/\I = \kk\dAA_n$, denoted by $\dA_n$ in our paper, is a hereditary algebra;
  if $\I = \rad^2(\kk\dAA_n)$, then $\kk\dAA_n/\I$, denoted by $\dA_n/\mathrm{r}^2$ in our paper, is gentle.
\end{itemize}
\end{example}

In \cite{BR1987}, Butler and Ringel described all indecomposable modules over string algebras and all Auslander-Reiten sequences by quiver methods. The following theorem is one of the works in their paper.

\begin{theorem}[{The irreducible morphisms over string algebras \cite[Section 3]{BR1987}}] \label{thm:BR1987}
Any Auslander-Reiten sequence in the finitely generated module category of string algebra is of the form
\[L \mathop{\longrightarrow}\limits^{\left[{^{f_1}_{f_2}}\right]} M_1\oplus M_2
    \mathop{\longrightarrow}\limits^{\left[ g_1\ g_2\right]} N, \]
where $L$, $M_1$, $M_2$ and $N$ are indecomposable, at least one of $M_1$ and $M_2$ is non-zero,
and $f_1$, $f_2$, $g_1$, $g_2$ are either irreducible morphisms or zero.
\end{theorem}

%

\subsection{Generalized diamond-stacked algebras }

In this part, we introduce a class of finite dimensional algebra whose quiver $\Q$ can be ``stacked'' by some diamond quivers, where {\defines diamond quiver} is a quiver with the underlying graph $\Square{i}{j} = $
\[\xymatrix{
  \vertex{i}{j+1} \ar@{-}[r]^{a_{i,j}^{\Upper}} \ar@{-}[d]_{a_{i,j}^{\Left}}
& \vertex{i+1}{j+1} \ar@{-}[d]^{a_{i,j}^{\Right}} \\
  \vertex{i}{j} \ar@{-}[r]_{a_{i,j}^{\Down}}
& \vertex{i+1}{j}
}\]
It can be used to describe the Auslander-Reiten quiver of string algebra.
Moreover, for any algebra $A=\kk\Q/\I$, we denote by $\e^A_v$ (or $\e_v$ for short) the idempotent of $A$ corresponding to the vertex $v\in \Q_0$.
For the sake of simplicity, the empty quiver $\varnothing = (\varnothing, \varnothing, \source, \target)$ is called a {\defines trivial diamond quiver} in this paper.
\checkLiu{In particular, a {\defines commutative diamond quiver} in given bound quiver $(\Q, \I)$ is a diamond quiver
\[\xymatrix{
  \vertex{i}{j+1} \ar[r]^{a_{i,j}^{\Upper}} \ar[d]_{a_{i,j}^{\Left}}
& \vertex{i+1}{j+1} \ar[d]^{a_{i,j}^{\Right}} \\
  \vertex{i}{j} \ar[r]_{a_{i,j}^{\Down}}
& \vertex{i+1}{j}
}\]
such that $a_{i,j}^{\Upper}a_{i,j}^{\Right} - a_{i,j}^{\Left}a_{i,j}^{\Down} \in \I$.
}

We say an underlying graph is a {\defines horizontal connection} of two diamond underlying graphs $\Square{i}{j}$ and $\Square{\imath}{\jmath}$, if it is of the form
\[\xymatrix{
  \vertex{i}{j+1} \ar@{-}[r]^{a_{i,j}^{\Upper}} \ar@{-}[d]_{a_{i,j}^{\Left}}
& {\vertex{i+1}{j+1} \atop = \vertex{\imath}{\jmath+1}} \ar@{-}[r]^{a_{\imath,\jmath}^{\Upper}}
  \ar@{-}[d]^{a_{\imath,\jmath}^{\Left}}_{a_{i,j}^{\Right}}
& \vertex{\imath+1}{\jmath+1} \ar@{-}[d]^{a_{\imath,\jmath}^{\Right}} \\
  \vertex{i}{j} \ar@{-}[r]_{a_{i,j}^{\Down}}
& {\vertex{i+1}{j} \atop = \vertex{\imath}{\jmath}} \ar@{-}[r]_{a_{\imath,\jmath}^{\Down}}
& \vertex{\imath+1}{\jmath}
}\]
That is, $a_{i,j}^{\Right} = a_{\imath,\jmath}^{\Left}$ (thus, we have $\vertex{i+1}{j+1} = \vertex{\imath}{\jmath+1}$ and $\vertex{i+1}{j} = \vertex{\imath}{\jmath}$ in this case, or equivalently, we have $\imath = i+1$ and $\jmath = j$).
We denote by $\Square{i}{j} \sqcup^{\rmH} \Square{\imath}{\jmath}$ the horizontal connection of $\Square{i}{j}$ and $\Square{\imath}{\jmath}$ for short.
Dually, a {\defines vertical connection} of $\Square{i}{j}$ and $\Square{\imath}{\jmath}$, say $\Square{i}{j} \sqcup^{\rmV} \Square{\imath}{\jmath}$, is the underlying graph given by $a_{i,j}^{\Upper} = a_{\imath, \jmath}^{\Down}$ with $\imath=i$ and $\jmath=j+1$.

Fixing $j\in \NN$, let $\pmb{x} = (x_{i,j})_{i\in\NN}$ be a sequence with $x_{i,j} \in (\NN\times\NN)\cup\{\varnothing\}$ and $\mathrm{supp}(\pmb{x}) =\{x_{i,j}\ne\varnothing\mid i\in\NN\}$ is a finite set.
Then $\pmb{x}$ provides a sequence $(\Square{i}{j}^{\star})_{i\in \NN}$ with
\[ \Square{i}{j}^{\star} = \begin{cases}
\Square{i}{j}, & \text{ if } x_{i,j} \ne \varnothing; \\
\varnothing, & \text{ if } x_{i,j} = \varnothing.
\end{cases}\]
Let
\begin{align}\label{mapping}
 h_{\pmb{x}}: \NN \times \NN \to \{\varnothing\}\cup \{\Square{i}{j}\}_{i\in\NN}, \
 h_{\pmb{x}}(\imath, \jmath) =
 \begin{cases}
 \Square{\imath}{j}^{\star}, & \text{ if } \jmath = j; \\
 \varnothing, & \text{otherwise}.
 \end{cases}
\end{align}
be a map which is induced by the above sequence $(\Square{i}{j}^{\star})_{j\in \NN}$ such that
if $h_{\pmb{x}}(i,j)\ne\varnothing$, $h_{\pmb{x}}(\imath,j)\ne\varnothing$, and $ i\ne \imath$,
then $h_{\pmb{x}}(i,j)\ne h_{\pmb{x}}(\imath,j)$.

Thus, we can define the {\defines horizontal connection of $m$ diamond underlying graphs $\Square{1}{j}^{\star}$, $\cdots$, $\Square{m}{j}^{\star}$ corresponding to $\pmb{x}$}
($m=|\mathrm{supp}(\pmb{x})|$) by
\[{\bigsqcup\limits_{i\in\NN}}{}^{\rmH}\ h_{\pmb{x}}(i,j) =
  {\bigsqcup\limits_{i=1}^{m}}{}^{\rmH}\ \Square{i}{j}^{\star} :=
   \bigcup_{t\in T}{\bigsqcup\limits_{i = r_t}^{s_t}}{}^{\rmH} \Square{i}{j}^{\star}, \]
where all unions ``$\cup$'' in the above formula are disjoint unions;
$T$ is some index set such that $ r_1\le s_1 < r_2 \le s_2 <\cdots <r_t\le s_t$;
and $\Square{i}{j}^{\star} \ne \varnothing$ holds for all $r_t\le i\le s_t$.

If there is no confusion, we call the horizontal connection corresponding to $\pmb{x}$ by the {\defines horizontal connection} for simplicity.

Furthermore, the {\defines vertical connection}
\[\left( {\bigsqcup\limits_{i=1}^{m}}{}^{\rmH} \Square{i}{j}^{\star} \right)
 {\bigsqcup}^{\rmV}
 \left( {\bigsqcup\limits_{i=1}^{m}}{}^{\rmH} \Square{i}{j+1}^{\star} \right) \]
of two horizontal connections ${\bigsqcup\limits_{i=1}^{m}}{}^{\rmH} \Square{i}{j}^{\star}$ and ${\bigsqcup\limits_{i=1}^{m}}{}^{\rmH} \Square{i}{j+1}^{\star}$ is defined in a similar way.

\begin{example} \label{exp:connection1} \rm 
Given two sequences
\begin{center}
  $\pmb{x}^{(\mathrm{I})} = (x_{i,0})_{i\in\NN}
   = ((0,0), \varnothing, (2,0), (3,0), \varnothing, \varnothing, \cdots)$
and

  $\pmb{x}^{(\mathrm{II})} = (x_{i,1})_{i\in\NN}
   = (\varnothing, (1,1), (2,1), \varnothing, \varnothing, \varnothing, \cdots)$.
\end{center}
Then we have
\begin{center}
  $\mathrm{im}(h_{\pmb{x}^{(\mathrm{I})}}) = \{\Square{0}{0}, \Square{2}{0}, \Square{3}{0}, \varnothing\}$
and

  $\mathrm{im}(h_{\pmb{x}^{(\mathrm{II})}}) = \{\Square{1}{1}, \Square{2}{1}, \varnothing\}.$
\end{center}

Let $\overline{\Q}^{(\mathrm{I})}
  = \mathop{\bigsqcup{}^{\rmH}}\limits_{(i,j)\in \NN\times\NN}h_{\pmb{x}^{(\mathrm{I})}}((i,j))$ and $\overline{\Q}^{(\mathrm{II})}
  = \mathop{\bigsqcup{}^{\rmH}}\limits_{(i,j)\in \NN\times\NN} h_{\pmb{x}^{(\mathrm{II})}}((i,j))$ be the horizontal
connections corresponding to $\pmb{x}^{(\mathrm{I})}$ and $\pmb{x}^{(\mathrm{II})}$ respectively. Then
$ \overline{\Q}^{(\mathrm{I})}$ is of the form
    \[\xymatrix{
       \vertex{0}{1} \ar@{-}[r] \ar@{-}[d] & \vertex{1}{1} \ar@{-}[d]
     & \vertex{2}{1} \ar@{-}[r] \ar@{-}[d] & \vertex{3}{1} \ar@{-}[r] \ar@{-}[d] & \vertex{4}{1} \ar@{-}[d] \\
       \vertex{0}{0} \ar@{-}[r] & \vertex{1}{0}
     & \vertex{2}{0} \ar@{-}[r] & \vertex{3}{0} \ar@{-}[r] & \vertex{4}{0}
    }\]
and $\overline{\Q}^{(\mathrm{II})}$ is of the form
    \[\xymatrix{
      & \vertex{1}{2} \ar@{-}[r] \ar@{-}[d] & \vertex{2}{2} \ar@{-}[r] \ar@{-}[d] & \vertex{3}{2} \ar@{-}[d] &  \\
      & \vertex{1}{1} \ar@{-}[r] & \vertex{2}{1} \ar@{-}[r] & \vertex{3}{1} &
    }\]

Then the vertical connection $\overline{\Q}^{(\mathrm{I})} \bigsqcup{}{^{\rmV}} \overline{\Q}^{(\mathrm{II})} = $
\begin{center}
  \begin{tikzpicture}
  \draw (0,0) node{\xymatrix{
     & \vertex{1}{2} \ar@{-}[r] \ar@{-}[d] & \vertex{2}{2} \ar@{-}[r] \ar@{-}[d] & \vertex{3}{2} \ar@{-}[d] &  \\
       \vertex{0}{1} \ar@{-}[r] \ar@{-}[d]_{a_{0,0}^{\Left}} & \vertex{1}{1} \ar@{-}[r] \ar@{-}[d]
     & \vertex{2}{1} \ar@{-}[r] \ar@{-}[d] & \vertex{3}{1} \ar@{-}[r] \ar@{-}[d]
     & \vertex{4}{1} \ar@{-}[d]^{a_{3,0}^{\Right}} \\
       \vertex{0}{0} \ar@{-}[r] & \vertex{1}{0}
     & \vertex{2}{0} \ar@{-}[r] & \vertex{3}{0} \ar@{-}[r] & \vertex{4}{0}
    }};
  \draw [line width=7pt][blue][opacity=0.25] (-3.70, 0   ) -- ( 3.70, 0   );
  \draw [line width=7pt][blue][opacity=0.25] (-3.70,-1.35) -- (-1.40,-1.35);
  \draw [line width=7pt][blue][opacity=0.25] (-0.40,-1.35) -- ( 3.70,-1.35);
  \end{tikzpicture}
\end{center}
\end{example}

\begin{definition}[Diamond-stacked quivers]  \rm
Let $\Q$ be a quiver.
\begin{itemize}
\item[(DS1)] $\Q$ is said to be a {\defines diamond-stacked quiver of the first kind {\rm(}$=$\PDS\ quiver{\rm)} with height $1$},
if its underlying graph $\overline{\Q}$ equals to
\begin{align}\label{PDS-1}
  \bigcup_{{\theta}=1}^{m}{\bigsqcup_{i=i_{\theta}}^{i_{\theta}+r_{\theta}}}{}^{\rmH}\ \Square{i}{j},
\end{align}
where
  $i_1 \le i_1+r_1 < i_2 \le i_2+r_2 < \cdots < i_m \le i_m+r_m$ is a sequence in $\NN$, $m$ and $j$ are integers in $\NN^+$, and $r_1, \cdots, r_m\in \NN$.

  Furthermore, $\Q$ is said to be a {\defines \PDS\ quiver with height $n$ $(\ge 1)$}, if $\overline{\Q}$ equals to a vertical connection
\begin{align}\label{PDS-n}
  {\bigsqcup_{\vartheta=0}^{n-1}}{}^{\rmV} \left(
  \bigcup_{{\theta}=1}^{m}{\bigsqcup_{i=i_{\theta}}^{i_{\theta}+r_{\theta}}}{}^{\rmH}\ \Square{i}{j+\vartheta}
  \right)
\end{align}
of some underlying graphs which are of the form (\ref{PDS-1}).

\item[(DS2)] $\Q$ is said to be a {\defines diamond-stacked quiver of the second kind quiver {\rm(}$=$\NPDS\ quiver{\rm)} }
    if its underlying graph $\overline{\Q}$ can be obtained by the following method.
    \begin{itemize}
      \item 
        For two sequences $\{v_i\}_{1\le i\le l}$ and $\{w_i\}_{1\le i\le l}$ of vertices in an underlying graph $\overline{\Q}'$ given by (\ref{PDS-n}), where $\{v_i\}_{1\le i\le l} \cap \{w_i\}_{1\le i\le l} = \varnothing$,
        we glue vertices $v_i$ and $w_i$, that is, $  \Q$ is the quiver given by the bound quiver of $\kk\Q'/\langle\e_{v_i}-\e_{w_i} \mid 1\le i\le l\rangle$.
    \end{itemize}
\end{itemize}
Furthermore, we say a quiver $\Q$ is a {\defines diamond-stacked {\rm(}$=$DS{\rm)} quiver} if
its underlying graph $\overline{\Q}$ is one of (DS1) and (DS2).
\end{definition}

\begin{example} \rm  \
\begin{itemize}

\item[(1)] A quiver with underlying graph $\overline{\Q}^{(\mathrm{I})} \bigsqcup{}{^{\rmV}} \overline{\Q}^{(\mathrm{II})}$ given in Example \ref{exp:connection1} is a \PDS\  quiver with height 2.

\item[(2)] Let $\Q$ be a quiver whose underlying graph is given by $\overline{\Q}^{(\mathrm{I})} \bigsqcup{}{^{\rmV}} \overline{\Q}^{(\mathrm{II})}$ in (1) such that $\vertex{0}{1}=\vertex{4}{0}$, $\vertex{0}{0}=\vertex{4}{1}$ and $a_{0,0}^{\Left} = a_{3,0}^{\Right}$, see \Pic \ref{fig:NPDSquiver}.
  Then $\Q$ is an \NPDS\  quiver.

\begin{figure}[H]
  \centering
  \begin{tikzpicture}[xscale=1.5, yscale=1]
  \draw ( 0.0,-1.0) to[out= 180, in= -90]
        (-1.5, 0.0) to[out=  90, in= 180]
        ( 0.0, 1.0) to[out=   0, in=  90]
        ( 1.0, 0.0) to[out= -90, in=   0]
        (-0.5,-2.0) to[out=-180, in= -90]
        (-3.0, 0.0) to[out=  90, in= 180]
        (-0.5, 2.0) to[out=   0, in=  45]
        ( 0.5, 0.0) to[out=  90, in= -10]
        ( 0.0, 0.5) to[out= 170, in=  90]
        (-1.0, 0.0) to[out=  90, in= -20]
        (-1.36,0.5) to[out= 160, in=  90]
        (-2.5,-0.5) to[out= -90, in= 135]
        (-2.0,-1.6);
  \draw ( 2.0, 0.0) to[out= 200, in=  10]
        ( 0.5,-0.5) to[out=  90, in= -90]
        ( 0.5, 0.0);
  \draw (-1.0,-0.8) to[out=-135, in= 180] (-0.2,-2.0);
  \draw[white] ( 0.85,-0.9) node{$\bullet$};
  \draw ( 2.0, 0.0) to[out=-120, in=   0] ( 0.0,-1.0);
  \fill[white] (-1.0,-0.8) circle(0.25);
         \draw (-1.0,-0.8) node{\small$\vertex{0}{0}$};
         \draw (-0.7,-0.8) node[above]{\small$(=\vertex{4}{1})$};
  \fill[white] ( 2.0, 0.0) circle(0.25);
         \draw ( 2.0, 0.0) node{\small$\vertex{1}{0}$};
  \fill[white] (-0.5, 2.0) circle(0.25);
         \draw (-0.5, 2.0) node{\small$\vertex{2}{0}$};
  \fill[white] (-2.1,-1.5) circle(0.25);
         \draw (-2.1,-1.5) node{\small$\vertex{3}{0}$};
  \fill[white] (-0.3,-2.0) circle(0.25);
         \draw (-0.3,-2.0) node{\small$\vertex{4}{0}$};
         \draw (-0.3,-2.0) node[below]{\small$(=\vertex{0}{1})$};
  \fill[white] (1.05,-0.4) circle(0.25);
         \draw (1.05,-0.4) node{\small$\vertex{1}{1}$};
  \fill[white] ( 0.7, 0.7) circle(0.25);
         \draw ( 0.7, 0.7) node{\small$\vertex{2}{1}$};
  \fill[white] (-1.36,0.5) circle(0.25);
         \draw (-1.36,0.5) node{\small$\vertex{3}{1}$};
  \fill[white] (-1.0, 0.0) circle(0.15);
         \draw (-1.0, 0.0) node{\small$\vertex{3}{2}$};
  \fill[white] ( 0.5, 0.0) circle(0.15);
         \draw ( 0.5, 0.0) node{\small$\vertex{2}{2}$};
  \fill[white] ( 0.5,-0.5) circle(0.15);
         \draw ( 0.5,-0.5) node{\small$\vertex{1}{2}$};
  \draw (-1.0,-1.5) node{$a_{0,0}^{\Left} = a_{3,0}^{\Right}$};
  \draw [line width=7pt][blue][opacity=0.25]
        ( 2.0, 0.0) to[out=-120, in=   0]
        ( 0.0,-1.0) to[out= 180, in= -90]
        (-1.5, 0.0) to[out=  90, in= 180]
        ( 0.0, 1.0) to[out=   0, in=  90]
        ( 1.0, 0.0) to[out= -90, in=   0]
        (-0.5,-2.0) to[out=-180, in= -90]
        (-3.0, 0.0) to[out=  90, in= 180]
        (-0.5, 2.0) ;
  \end{tikzpicture}
  \caption{The \NPDS\  underlying graph}
  \label{fig:NPDSquiver}
\end{figure}
\end{itemize}
\end{example}

\begin{definition} \rm \label{def:GDS}
Let $A=\kk\Q/\I$ be the algebra of a bound quiver $(\Q,\I)$ whose ideal $\I$ is admissible. $A$ is said to be
\begin{itemize}
  \item[(1)] a {\defines \PDS\ {\rm(}resp. \NPDS{\rm)} algebra} if $\Q$ is a \PDS\ {\rm(}resp. \NPDS{\rm)} quiver;
  \item[(2)] a {\defines \DS\ algebra} if it is either a \PDS\ algebra or a \NPDS\ algebra;
  \item[(3)] a {\defines generalized diamond-stacked {\rm(}$=$\GDS{\rm)} algebra} if there exists a DS algebra $\Lambda$ such that the following conditions hold.
      \begin{itemize}
        \item $\Q$ equals to the quiver given by the bound quiver of $\alg/\alg\e\alg$ ($\e$ is an idempotent of $\alg$);
        \item All diamond subquivers of $\Q$ are commutative in $A$.
      \end{itemize}
\end{itemize}
\end{definition}

Let $(\Q,\I)$ be a bound quiver of $A=\kk\Q/\I$ whose ideal $\I$ is admissible.
Then a quiver $\Q^{\g}$ obtained by {\defines \gluing vertices $v_i, w_i\in \Q_0$} ($i\in L, L$ is an index set) is the original quiver of $A^{\g}=A/\langle\e_{v_i}-\e_{w_i} \mid i\in L\rangle$,
and, naturally, any path on $\Q$ can be seen as a path on $\Q^{\g}$ (thus we have $v_i=w_i$ on $\Q^{\g}_0$).
During the above process, some paths on $\Q^{\g}$ crossing $v_i$ $(=w_i)$ are newly acquired by \gluing $v_i$ and $w_i$, and they equal to zero in $A^{\g}$.

For any $\mathit{\Lambda}=\kk\Q_\itLamb/\I_\itLamb$ with the generator of $\I_\itLamb$ being one of following forms:
\begin{itemize}
  \item $\e_{v_i}-\e_{w_i}$ ($i\in L$);
  \item path, say zero relation, of length $\ge 2$;
  \item a sum $\sum\limits_{t=1}^{m}\wp_t$ of $m$ ($m\ge 2$) paths of length $\ge$ 2.
\end{itemize}
Let $\idem_{\I_\itLamb}$, $\z_{\I_\itLamb}$ and $\s_{\I_\itLamb}$ respectively be the set of all generators of $\I_{\itLamb}$ of the first, second and third form as above, then  $\I_\itLamb = \langle \idem_{\I_\itLamb} \cup \z_{\I_\itLamb}\cup \s_{\I_\itLamb}\rangle$.
We denote by $\itLamb(\frakp(v_{\imath})\mid \imath\in I) $ ($I$ is an index set such that $\{v_{\imath}\mid \imath \in I\} \subseteq (\Q_\itLamb)_0$ and $I\supseteq\ L$) the $\kk$-vector space
\begin{align}
  \itLamb(\frakp(v_{\imath}) \mid \imath\in I) = \spa(S) / \J
\end{align}
\[ \left( S := \big((\Q_\itLamb)_{\ge 0}\backslash \{w_i \mid i\in I\}\big)
     \bigcup\bigcup_{\imath\in I}\frakp(v_{\imath}) \right), \]
where
\begin{itemize}
  \item $\frakp(v_{\imath})$ is the set whose element is \checkLiu{of the form $\wp_1\otimes\wp_2$, where} $\wp_1$ and $\wp_2$ are non-zero paths in $(\Q_{\itLamb}, \I_{\itLamb})$ with $\target(\wp_1), \source(\wp_2)\in\{v_{\imath}, w_{\imath}\}\subseteq(\Q_\itLamb)_0$;
  \item and $\J = \big\langle \s_{\I_\itLamb}\cup\big(\z_{\I_\itLamb} \big\backslash \bigcup\limits_{i\in I}\frakp(v_i)\big) \big\rangle$.
\end{itemize}
Notice that $\spa(S)$ is a $\kk$-algebra because the canonical bijection
\[ \sigma: S = \big((\Q_\itLamb)_{\ge 0}\backslash \{w_i\mid i\in L\}\big)
     \bigcup\bigcup_{\imath\in I}\frakp(v_{\imath}) \to (\Q^{\g}_\itLamb)_{\ge 0} \]
exists. Therefore, $\itLamb(\frakp(v_i)\mid i\in I)\cong \kk\Q^{\g}_\itLamb/\J$ is a $\kk$-algebra which can be obtained by gluing $v_i$ and $w_i$.

\begin{example}\rm
Let $A= \kk\Q/\I$ with $\Q = \xymatrix{1 \ar[r]^{a} & 2 \ar[r]^b & 3\ar[r]^c & 4}$ and $\I=\langle ab\rangle$.
Then the original quiver $\Q^{\g}$ of $A^{\g} = A/\langle\e_2-\e_4\rangle$ is
\[\xymatrix{1 \ar[r]^{a} & 2 \ar[r]^b & 3 \ar@/^1pc/[l]^c. }\]
Notice that $cb = 0$ in $A^{\g}$ because the preimage of $cb+\langle\e_2-\e_4\rangle$ under the canonical epimorphism $\nat: A \to A^{\g}$, $\nat(x) = x+\langle\e_2-\e_4\rangle$ is the path $cb$ being zero in $A$.
We have $\frakp(2)=\frakp(4) = \{a, b, c, ab, bc, cb, abc, bcb, cbc, abcb, bcbc, cbcb, \cdots\}$, and further
\begin{align}
       A^{\g}(\frakp(2))
 =\ & \spa(\{\e_1, \e_2, \e_3, a, b, c, bc\} \cup \frakp(2) )/\I^{\g} \nonumber \\
 =\ & \kk\e_1 + \kk\e_2 + \kk\e_3 + \kk a +\kk b + \kk c + \nonumber \\
    & + \kk ab + \kk bc + \kk cb + \kk abc + \kk bcb + \kk cbc + \cdots \nonumber.
\end{align}
\checkLiu{Note that the ideal $\I^{\g} = 0$}.
\end{example}

\begin{definition}\rm (\Gluing algebras) \label{Gluing algebras}
Let $A = \kk\Q/\I$ be an algebra with admissible ideal $\I$.
\begin{itemize}
  \item[(1)] An {\defines incomplete \gluing algebra} of $A$, say $A^{\g}_X$, is a quotient $A/\langle X\rangle$ with $X=\{\e_{v_i}-\e_{w_i} \mid 1\le i\le l\} \cup \{v_j \mid j\in J\}$
      ($J$ is an index set containing no $1, 2, \cdots, l$).
      We call $X$ (or $J\cup\{1,\cdots,l\}$) the {\defines \gluing index} of $A$.
  \item[(2)] Furthermore, for any subset $\frakS = \bigcup\limits_{\imath\in I}\frakp(v_{\imath}) \supseteq \bigcup\limits_{1\le i\le l}\frakp(v_i)$ ($I$ is an index set) of $(\Q_{X}^{\g})_{\ge 0}$,
    where $\Q_{X}^{\g}$ is given by the bound quiver of $A_X^{\g}$,
    we call $A^{\g}_X(\frakS)$ a {\defines \gluing algebra} of $A$ induced by $A^{\g}_X$ and $\frakS$.
    The set $\frakS$ is called the {\defines supplement} of $A^{\g}_X$ decided by $\{v_{\imath}\mid \imath \in I\}$.
\end{itemize}
\end{definition}

\begin{remark} \rm \label{rmk:gluing}\
\begin{itemize}
\item[(1)] If $X=\{\e_{v_i}-\e_{w_i}\mid 1\le i\le l\}$, then $A^{\g}_X=A^{\g}=A/\langle\e_{v_i}-\e_{w_i} \mid 1\le i\le l\rangle$.
  \item[(2)] If $X = \{v_j\mid 1\le j\le m\}$, then $A^{\g}_X = A/A\e A$, where $\e=\e_{v_1}+\cdots+\e_{v_m}$.
    In this case, the quiver $\Q^{\g}_X$ given by the bound quiver of $A^{\g}_X$
    can be obtained by deleting all vertices $v_j$ of $\Q$.
  \item[(3)] We allow $X$ and $\frakS$ to be an empty set.
    \begin{itemize}
      \item If $X = \varnothing$, then $A^{\g}_{\varnothing}(\frakS)=A(\frakS)$.
      \item If $\frakS = \varnothing$, then $A^{\g}_{X}(\varnothing)=A/\langle X\rangle$.
      \item If $X = \varnothing$ and $\frakS = \varnothing$, then $A^{\g}_{X}(\frakS)=A$.
    \end{itemize}
\end{itemize}
\end{remark}

The following lemma are very useful which are needed later on.

\begin{lemma} \label{lemm:ideal}
Let $A = \kk\Q/\I$ be a finite-dimensional algebra with ideal $\I = \langle r_1, \cdots, r_m\rangle$,
and $\J = \langle s_1, \cdots, s_n \rangle$ be an ideal of $\kk\Q$ {\rm(}$s_1, \cdots, s_n$ are paths on $\Q${\rm)},
then
\[ A/\langle s_j + \I \mid 1\le j\le n\rangle \ (= (\kk\Q/\I)/\langle s_j + \I \mid 1\le j\le n\rangle )\ \cong \kk\Q/\langle \I, \J\rangle. \]
\end{lemma}

\begin{proof}
Let $\langle s_j + \I \mid 1\le j\le n\rangle = \J'$, and define
\begin{align}\label{formula:ideal}
 f: (\kk\Q/\I)/\J' \to \kk\Q/\langle\I,\J\rangle, \\
 (a+\I)+\J' \mapsto a+\langle\I,\J\rangle. \nonumber
\end{align}
Then for any $(a+\I)+ \J' = (a'+\I)+ \J'$, we have $(a+\I) - (a'+\I) = (a-a')+\I$, as an element in $A$, belongs to $\J'$, that is,
$(a-a')+\I \in \J'$. So $(a-a')\in \langle\I,\J\rangle$
and $(a+\langle\I,\J\rangle) - (a'+\langle\I,\J\rangle) = (a-a')+\langle\I,\J\rangle = \langle\I,\J\rangle$, $f$ is well-defined.

Next, we show that $f$ is bijective. It is easy to see that $f$ is an epimorphism because arbitrary element $a + \langle\I,\J\rangle$ in $\kk\Q/\langle\I,\J\rangle$ has a preimage $(a + \I) + \J'$.
On the other hand, for any $(x + \I) + \J' \in \ker f$, we have $x \in \langle\I,\J\rangle$, that is,
\begin{align}
x & = \lambda x_{\I} + \mu x_{\J}
\ (\text{ for some } \lambda,\mu \in \kk \text{ and } x_{\I}\in\I, x_{\J}\in\J).  \nonumber
\end{align}
Then we have $x+\I = \mu x_{\J} + \I = \mu(x_{\J} + \I) \in \J'$.
Thus, $(x+\I)+\J'$ is zero in $A/\J$, and so $f$ is monomorphic.
\end{proof}
\section {Characterization of
generalized diamond-stacked algebras}
\begin{notation}\color{red} \rm
\checkLiu{For any algebra $A=\kk\Q/\I$ with $\I=\langle r_i \mid i\in I \rangle$ and arbitrary ideal $\J=\langle s_j \mid j\in J\rangle$ of $\kk\Q$ ($I$ and $J$ are index sets and $\{r_i \mid i\in I\}\cap \{s_j\mid j\in J\}$ $=\varnothing$),
we denote by $A/\J$ the quotient $A/\langle s_j+\I \mid j\in J\rangle$ of $A$ for short. }
\end{notation}


\begin{lemma} \label{lemm:supplement}
Let $A = \kk\Q/\I$ be a finite-dimensional algebra with ideal $\I$
and $\J = \langle s_1, \cdots, s_n \rangle$ be an ideal of $\kk\Q$.
If $\frakp(v)$, the set of all paths on $(\Q, \I)$ crossing $v\in \Q_0$, contains no $s_j$ $(1\le j\le n)$,
then $\frakp(v)$ can be seen as a set of some paths on $(\Q, \langle\I,\J\rangle)$, and
\[(A/\J)(\frakp(v)) \cong A(\frakp(v))/\J. \]
\end{lemma}

\begin{proof}
Let $\langle s_j + \I \mid 1\le j\le n\rangle = \J'$, and define
\[f: (A/\J)(\frakp(v)) \to A(\frakp(v))/\J, \]
\begin{center}
  $((a+\I)+ \J') \pmb{+} \sum\limits_{\wp\in\frakp(v)} k_{\wp}((\wp+\I)+\J')
  \mapsto \big((a+\I) \pmb{+} \sum\limits_{\wp\in\frakp(v)} k_{\wp}(\wp+\I)\big)+\J'$
\end{center}
($k_{\wp} \in \kk$ holds for all $\wp \in (\Q,\I)$, and ``$\pmb{+}$'' and ``$+$'' are two different types of addition).
The proof of the well-definability of $f$ is similar to that of (\ref{formula:ideal}).

Next, we show that $f$ is bijective. First of all, it is easy to see that $f$ is an epimorphism
because arbitrary element $\big((a+\I) \pmb{+} \sum\limits_{\wp\in\frakp(v)} k_{\wp}(\wp+\I)\big)+\J'$ in $A(\frakp(v))/\J$
has a preimage $((a+\I) + \J') \pmb{+} \sum\limits_{\wp\in\frakp(v)} k_{\wp}((\wp+\I)+\J')$ in $(A/\J)(\frakp(v))$.

\checkLiu{
On the other hand, for any $((x+\I)+ \J') \pmb{+} \sum\limits_{\wp\in\frakp(v)} k_{\wp}((\wp+\I) + \J')$ in $\ker f$, we have
$(x+\I) + \Big(\sum\limits_{\wp\in\frakp(v)} (k_{\wp}\wp+ \I)\Big) \in \J'.$
Without loss of generality, assume that $x=x_1+x_2$ with $x_1=\sum\limits_{\tilde{\wp}\notin\frakp(v)} \lambda_{\tilde{\wp}}\tilde{\wp}$ and $x_2=\sum\limits_{\wp\in\frakp(v)} \mu_{\wp}\wp$ ($\lambda_{\tilde{\wp}}, \mu_{\wp} \in \kk$). Then
\begin{align} \label{formula:kerelem}
 (x_1+\I) + \Sigma \in \J' = \langle s_j + \I \mid 1\le j\le n\rangle
\end{align}
where $x_1+\I \notin \frakp(v)+\I$ and $\Sigma=\sum\limits_{\wp\in\frakp(v)} ((k_{\wp}+\mu_{\wp})\wp+ \I) \in \frakp(v)+\I$.
Since $\frakp(v)$ dose not {contain} any $s_j$ ($1\le j\le n$),
the sum $\Sigma$ is zero in $A/\J=(\kk\Q/\I)/\langle r+\J\mid r\in \I\rangle$.
Then $x_1+\I\in \J'$. It follows that $\bar{x} = (x_1+\I)+\Sigma + \J'$ is zero in $(A/\J)(\frakp(v))$. }
\end{proof}

Next, we provide a lemma and a proposition which are used to show our main result.

\begin{lemma} \label{lemm:NPDSisquotient}

Any \NPDS\ algebra $A=\kk\Q/\I$ with admissible ideal $\I$ is isomorphic to a quotient of a \gluing algebra of some \PDS\  algebra,
where any generator of $\I$ is either a zero relation or a commutative relation given by a diamond quiver.
\end{lemma}

\begin{proof}
Without loss of generality, we give two assumptions as follows:
\begin{itemize}
  \item $\Q$ is obtained by \gluing vertices $v_i, w_i$ ($1\le i \le l$) of some quiver \PDS\ $\Q'$
    (by the definition of \NPDS\ quiver);
  \item \checkLiu{All zero relations in $\I$ are paths $r_1$, $\ldots$, $r_n$, and the other relations in $\I$ are $c_1$, $\ldots$, $c_m$.}
\end{itemize}
Then every $r_k$ can be seen as a path $r'_k$ on $\Q'$ \checkLiu{which is either a relation of length $\ge 2$ or a trivial path $0$}.
\checkLiu{To be more precise, there are two situations: }
\begin{itemize}
  \item[(i)] \checkLiu{$r_k$ is a path of length $\ge 2$ on $\Q'$;}
  \item[(ii)] \checkLiu{$r_k$ is a new path on $\Q$ which is obtained by gluing vertices of $\Q$,
    then $r_k$ can not be seen as a non-zero path in $\Q'$, and we define $r_k'=0$ in this situation.}
\end{itemize}
Let $\I' = \langle r'_k \ne 0 \mid 1\le k\le n \rangle$ and $A'=\kk\Q'/\I'$.
It is easy to see that $(\Q, \langle r_k\mid 1\le k\le n \rangle)$ is the bound quiver of $ A'^{\g}_{\{\e_{v_i}-\e_{w_i} \mid 1\le i\le l\}} (=A'^{\g})$, that is,
\begin{align}
A'^{\g} = A'/\langle \e_{v_i}-\e_{w_i} \mid 1\le i\le l \rangle \cong \kk\Q/\langle r_k \mid 1\le k\le n \rangle. \nonumber
\end{align}
Consider the ideal $\J$ of $\kk\Q$ whose generators are given by
\begin{itemize}
  \item paths $\wp_1$, $\cdots$, $\wp_m$ of length $\ge 2$ belong to $\I$
    such that any $\wp_t$ ($1\le t\le m$) crosses $v_i$ ($=w_i$);
  \item \checkLiu{and the other relations $c_1, \ldots, c_m$ in $\I$}.
\end{itemize}
Obviously, the generators of $\I$ are divided to two parts, one is corresponded by the elements in $\{r_k'\ne 0\mid 1\le k\le n\}$,
the other is corresponded by the generators of $\J$. Then
\begin{align} \label{formula:NPDSisquotient}
A = \kk\Q/\I
& \cong A'^{\g}\Big(\bigcup_{i=1}^{l}\frakp(v_i)\Big) \Big/ \J
  = (\kk\Q'/\I')^{\g}\Big(\bigcup_{i=1}^{l}\frakp(v_i)\Big) \Big/ \J \\
& \left( \cong (\kk\Q')^{\g}_{\{\e_{v_i}-\e_{w_i}\mid 1\le i\le l\}}
  \Big(\bigcup_{i=1}^{l}\frakp(v_i)\Big) \Big/ \langle \{r_k'\mid 1\le k\le n\} \cup \mathcal{J} \rangle \right), \nonumber
\end{align}
where $A'^{\g}\Big(\bigcup\limits_{i=1}^{l}\frakp(v_i)\Big)$ is a gluing algebra of the \PDS\ algebra $A'$.
\end{proof}

\begin{proposition} \label{prop:GDSisquotient}
Let $A=\kk\Q/\I$ be a \GDS\ algebra with admissible ideal $\I$ whose generator is either a zero relation or a commutative relation.
Then $A$ is isomorphic to a quotient of a \gluing algebra of some \PDS\  algebra.
\end{proposition}

\begin{proof}
By the definition of \GDS\ algebra (see Definition \ref{def:GDS}),
there exists a DS algebra $\alg = \kk\Q_{\alg}/\I_{\alg}$ and some idempotent $\e = \sum\limits_{i\in \frakI}\e_i$
($\e_{i}$ is the idempotent corresponding to some vertex $i \in \Q_0$, and $\frakI \subseteq \Q_0$ is an index set) such that
$\Q = \Q_{\alg/\alg\e\alg}$. Then $A = \kk\Q/\I = \kk\Q_{\alg/\alg\e\alg}/\I$.
Assume that $\I$ is generated by $l$ commutative relations $c_1$, $\cdots$, $c_l$ and $m$ zero relations $z_1$, $\cdots$, $z_m$.
By Lemma \ref{lemm:ideal}, we obtain
\begin{align}
A & = \kk\Q_{\alg/\alg\e\alg} \big/ \langle c_i, z_j \mid 1\le i\le l, 1\le j\le m \rangle \nonumber \\
  & \cong (\kk\Q_{\alg/\alg\e\alg}/\langle c_i\mid 1\le i\le l\rangle)
    \big/ \langle z_j\mid 1\le j\le m\rangle, \nonumber
\end{align}
where $\kk\Q_{\alg/\alg\e\alg}/\langle c_i\mid 1\le i\le l\rangle$ is a \GDS\ algebra, denoted by $B$,
whose ideal is admissible and generated by commutative relations.
Thus, $(\Q_{\alg/\alg\e\alg}, \langle c_i\mid 1\le i\le l\rangle)$ is the bound quiver of $B$ and $A\cong B/\langle z_j\mid 1\le j\le m\rangle$.

If $B$ is a \PDS\ algebra, we are done.

Otherwise, it is easy to see that there exists a DS algebra $D=\kk\Q_D/\I_D$
and an idempotent $\e^D = \sum\limits_{k=1}^n \e^D_k$ (all $\e^D_k$ are primitive idempotents) in $D$
such that $\Q_{\alg/\alg\e\alg}$ equals to the quiver given by the bound quiver of $\kk\Q_D/\langle \e^D \rangle$.
To be more precise, the bound quiver of $\kk\Q_D/\langle \e^D \rangle$ can be given by deleting all vertices corresponded by $\e^D_k$.
During the above process, $\Q_D$ changes to $\Q_{\alg/\alg\e\alg}$, and some paths $p_1$, $\cdots$, $p_n$ on $\Q_D$ are zero in $\kk\Q_D/\langle \e^D \rangle$.
Let $V_k$ be the set of all vertices on $p_k$ ($1\le k\le n$) and $\frakS = \bigcup\limits_{k=1}^n\bigcup\limits_{v\in V_k}\frakp(v)$. Now, we have two cases as follows.
\begin{itemize}
\item[Case] 1. $\Q_D$ is a \PDS\ quiver. In this case,
\[ \kk\Q_{\alg/\alg\e\alg} \cong \big(\kk\Q_D/\langle \e^D \rangle\big)(\frakS), \]
Thus,
\begin{align}
B & = \kk\Q_{\alg/\alg\e\alg}/\langle c_i\mid 1\le i\le l \rangle && \nonumber \\
& \cong \big(\kk\Q_D/\langle \e^D \rangle\big)(\frakS) \Big/ \langle c_i\mid 1\le i\le l \rangle && \nonumber \\
& \cong \big((\kk\Q_D/\langle \e^D \rangle) \big/ \langle c_i\mid 1\le i\le l \rangle\big)(\frakS)
  && \text{ (by Lemma \ref{lemm:supplement}) } \nonumber \\
& \cong \big(\kk\Q_D/\langle \e_k^D, c_i\mid 1\le k\le n, 1\le i\le l \rangle\big)(\frakS)
  && \text{ (by Lemma \ref{lemm:ideal}) } \nonumber \\
& \cong \big((\kk\Q_D/\langle c_i\mid 1\le i\le l \rangle) / \langle\e^D\rangle \big)(\frakS)
  && \text{ (by Lemma \ref{lemm:ideal}) } \nonumber \\
& = \big(\kk\Q_D/\langle c_i\mid 1\le i\le l \rangle\big)^{\g}_{\{\e^D_k \mid 1\le k\le n\}}(\frakS). && \nonumber
\end{align}
The above formula shows that $B$ is a \gluing algebra of a \PDS\ algebra $\kk\Q_D/\langle c_i\mid 1\le i\le l\rangle$.

\item[Case] 2. $\Q_D$ is a \NPDS\ quiver. Then there is a \PDS\ algebra $\widehat{D} = \kk\Q_{\widehat{D}}/\langle H \rangle$
($H=\{\e^{\widehat{D}}_{v_h}-\e^{\widehat{D}}_{w_h} \mid 1\le h\le \hbar\}$)
such that $\Q_D$ equals to the quiver given by the bound quiver of $\widehat{D}$ (see the definition of \NPDS\ quiver).
Thus, $\Q_{\alg\e\alg}$ can be given by $\Q_{\widehat{D}}$ through following steps.
\begin{itemize}
  \item[]Step 1. \Gluing the vertices $v_h$ and $w_h$ ($1\le h\le \hbar$);
  \item[]Step 2. and deleting all vertices $\e^D_k$ ($1\le k\le n$) which naturally are seen as the vertices $\e^{\widehat{D}}_k$ of $\Q_{\widehat{D}}$.
\end{itemize}
Let $\frakS' = \frakS\bigcup\bigcup\limits_{1\le h\le \hbar}\frakp(v_h)$. Then
\[ \kk\Q_{\alg\e\alg} \cong \big(\kk\Q_{\widehat{D}}/\langle\e^{\widehat{D}}, H\rangle\big)(\frakS'), \]
and so
\begin{align}
B & = \kk\Q_{\alg/\alg\e\alg}/\langle c_i\mid 1\le i\le l \rangle && \nonumber \\
& \cong \big(\kk\Q_{\widehat{D}}/\langle \e^{\widehat{D}}, H \rangle\big)(\frakS')
  \Big/ \langle c_i\mid 1\le i\le l \rangle && \nonumber \\
& \cong \big((\kk\Q_{\widehat{D}}/\langle \e^{\widehat{D}}, H \rangle)
  \big/ \langle c_i\mid 1\le i\le l \rangle\big)(\frakS')
  && \text{ (by Lemma \ref{lemm:supplement}) } \nonumber \\
& \cong \big(\kk\Q_{\widehat{D}}/\langle \e_k^{\widehat{D}}, H, c_i\mid 1\le k\le n, 1\le i\le l \rangle\big)(\frakS')
  && \text{ (by Lemma \ref{lemm:ideal}) } \nonumber \\
& \cong \big((\kk\Q_{\widehat{D}}/\langle c_i\mid 1\le i\le l \rangle) / \langle\e^{\widehat{D}}, H\rangle \big)(\frakS')
  && \text{ (by Lemma \ref{lemm:ideal}) } \nonumber \\
& = \big(\kk\Q_{\widehat{D}}/\langle c_i\mid 1\le i\le l \rangle\big)^{\g}_{\{\e^{\widehat{D}}_k \mid 1\le k\le n\}\cup H}(\frakS). && \nonumber
\end{align}
The above formula shows that $B$ is a \gluing algebra of a \PDS\ algebra $\kk\Q_{\widehat{D}}/\langle c_i\mid 1\le i\le l\rangle$.
\end{itemize}
Therefore, $A \cong B/\langle z_j\mid 1\le j\le m\rangle$ is a quotient of a gluing algebra $B$ of some \PDS\ algebra.
\end{proof}

\begin{remark} \rm
Note that the proof of the case for $B$ being a \NPDS\ algebra is similar to that of Case 2.
\end{remark}
\begin{definition}[Tensor algebras]\rm
Let $A$ and $B$ be two finite-dimensional algebras whose bound quivers are $(\Q_A, \I_A)$ and $(\Q_B, \I_B)$, respectively. The {\defines $\kk$-tensor algebra} of $A$ and $B$ is the tensor $A\otimes_{\kk} B$ of $A$ and $B$ as $\kk$-vector spaces.\end{definition}
By \cite[Proposition 3]{Her2008}, the quiver $\Q_{A\otimes_{\kk}B}$ of $A\otimes_{\kk} B = \kk\Q_{A\otimes_{\kk}B} / \I_{A\otimes_{\kk}B}$ is the quadruple $((\Q_{A\otimes_{\kk}B})_0, (\Q_{A\otimes_{\kk}B})_1, \source, \target)$ given by
\begin{itemize}
    \item $(\Q_{A\otimes_{\kk}B})_0 = (\Q_A)_0\times(\Q_B)_0 = \{(i,j): i \in (\Q_A)_0, j \in (\Q_B)_0\}$,
    \item $(\Q_{A\otimes_{\kk}B})_1=\{e_x\otimes y : x \in (\Q_A)_0, y\in (\Q_B)_1 \}$ $\cup $  $\{x\otimes e_y :x\in (\Q_A)_1, y\in (\Q_B)_0\}$,
    \item $\source: (\Q_{A\otimes_{\kk}B})_1 \to (\Q_{A\otimes_{\kk}B})_0$ is the function respectively sending any arrows $e_x\otimes y$ and $x\otimes e_y$ to the elements $(x,\source(y))$ and $(\source(x), y)$,
    \item and $\target: (\Q_{A\otimes_{\kk}B})_1 \to (\Q_{A\otimes_{\kk}B})_0$  is the function respectively sending any arrows $e_x\otimes y$ and $x\otimes e_y$ to the elements $(x, \target(y))$ and $(\target(x), y)$.
\end{itemize}
The ideal $\I_{A\otimes_{\kk}B}$ is induced by $\I_A$, $\I_B$ and the properties of tensor product,
to be more precise, $\I_{A\otimes_{\kk}B}$ is generated by the following relations.
\begin{itemize}
  \item $e_x \otimes \wp_B$ ($x\in (\Q_A)_0$), where $\wp_B$ is a generator of $\I_B$,
  \item $\wp_A \otimes e_y$ ($y\in (\Q_B)_0$), where $\wp_A$ is a generator of $\I_A$,
  \item $(e_i\otimes\beta)(\alpha\otimes e_{j'}) - (\alpha\otimes e_{i'})(e_j\otimes \beta)$,
    where $\alpha: i\to j$ is an arrow in $(\Q_A)_1$ and $\beta: i'\to j'$ is an arrow in $(\Q_B)_1$,
    that is, any subquiver
    \[\xymatrix{
    (i,j') \ar[r]^{\alpha\otimes e_{j'}} & (j,j')  \\
    (i,i') \ar[u]^{e_i\otimes\beta} \ar[r]_{\alpha\otimes e_{i'}} & (j,i') \ar[u]_{e_j\otimes \beta} }\]
    of $\Q_{A\otimes_{\kk}B}$ is commutative.
\end{itemize}
In particular, the {\defines enveloping algebra} $A^e$ of finite-dimensional algebra $A$ is the $\kk$-tensor $A\otimes_{\kk}A^{\op}$, where $A^{\op}$ is the opposite algebra of $A$.

\begin{example}\rm
Denote by $A_n = \kk\Q/\I$ the algebra of type $\A$ with $n$ vertices, i.e.,
the underlying graph of its quiver $\Q$ is the type of $\A_n = $
\[ \xymatrix{1 \ar@{-}[r] & 2 \ar@{-}[r] & \cdots \ar@{-}[r] & n.} \]
Then the underlying graph of the tensor $A_m \otimes_{\kk} A_n$ ($=A_{m\otimes n}$ for short) of two algebras $A_m$ and $A_n$ is a \PDS\ algebra whose underlying graph is of the form shown in \Pic \ref{fig:Am-tensor-An} (1).
Thus, $A_{m\otimes n}$ is a \PDS\  algebra.
\begin{figure}[H]
\centering
\begin{tikzpicture}
\draw (0,0) node{\tiny$\xymatrix@=0.6cm@R=0.6cm{
 (1,n)  \ar@{-}[r]\ar@{-}[d] & (2,n)  \ar@{-}[r]\ar@{-}[d] & \cdots \ar@{-}[r]\ar@{-}[d] & (m,n)  \ar@{-}[d] \\
 \vdots \ar@{.}[r]\ar@{-}[d] & \vdots \ar@{.}[r]\ar@{-}[d] & \vdots \ar@{.}[r]\ar@{-}[d] & \vdots \ar@{-}[d] \\
 (1,2)  \ar@{-}[r]\ar@{-}[d] & (2,2)  \ar@{-}[r]\ar@{-}[d] & \cdots \ar@{-}[r]\ar@{-}[d] & (m,2)  \ar@{-}[d] \\
 (1,1)  \ar@{-}[r]           & (2,1)  \ar@{-}[r]           & \cdots \ar@{-}[r]           & (m,1)
}$};
\draw (0,-2.5) node{$(1)$};
\end{tikzpicture}
\ \
\begin{tikzpicture}
\draw (0,0) node{\tiny$\xymatrix@=0.6cm@R=0.6cm{
 (1,n)  \ar@{->}[r]\ar@{<-}[d] & (2,n)  \ar@{->}[r]\ar@{<-}[d] & \cdots \ar@{->}[r]\ar@{<-}[d] & (m,n)  \ar@{<-}[d] \\
 \vdots \ar@{.}[r] \ar@{<-}[d] & \vdots \ar@{.}[r] \ar@{<-}[d] & \vdots \ar@{.}[r] \ar@{<-}[d] & \vdots \ar@{<-}[d] \\
 (1,2)  \ar@{->}[r]\ar@{<-}[d] & (2,2)  \ar@{->}[r]\ar@{<-}[d] & \cdots \ar@{->}[r]\ar@{<-}[d] & (m,2)  \ar@{<-}[d] \\
 (1,1)  \ar@{->}[r]            & (2,1)  \ar@{->}[r]            & \cdots \ar@{->}[r]            & (m,1)
}$};
\draw (0,-2.5) node{$(2)$};
\end{tikzpicture}
\caption{The underlying graph of $A_m \otimes_{\kk} A_n$ and the quiver of $\protect\dA_m\otimes_{\kk}\protect\dA_n$}
\label{fig:Am-tensor-An}
\end{figure}
\end{example}

\begin{definition}\rm
We call that a \PDS\  algebra (resp. \PDS\  quiver) is {\defines standard} if its underlying graph is of the form shown in \Pic \ref{fig:Am-tensor-An}(1).
\end{definition}

\begin{remark} \rm
Tensor algebra $\dA_m\otimes_{\kk}\dA_n$ is a standard \PDS\ algebra, whose quiver is shown in \Pic \ref{fig:Am-tensor-An} (2).
\end{remark}

The following lemma gives a {more} ``standard'' description of \NPDS\  algebras than Lemma \ref{lemm:NPDSisquotient}.

\begin{lemma} \label{lemm:NPDSisquotient2}
\begin{itemize}
  \item[ ]
  \item[\rm(1)]
    Let $A$ be a \PDS\ algebra and $A_X^{\g}(\frakS_1)$ a \gluing algebra of $A$
    {\rm(}$\frakS_1$ is a supplement of $A_X^{\g}${\rm)}.
    For arbitrary quotient $A_X^{\g}(\frakS_1)/\I$ of $A_X^{\g}(\frakS_1)$,
    there is a standard \PDS\ algebra $B$ such that
    \[ A_X^{\g}(\frakS_1)/\I \cong B_Y^{\g}(\frakS_2)/\J, \]
    where $B_Y^{\g}(\frakS_2)$ is a \gluing algebra of $B$ and $\frakS_2$ is a supplement of $B_Y^{\g}$.
  \item[\rm(2)] Furthermore, any \NPDS\  algebra $A=\kk\Q/\I$ with admissible ideal $\I$ is isomorphic to a quotient of a \gluing algebra of some standard \PDS\  algebra.
\end{itemize}
\end{lemma}

\begin{proof}
(1) Indeed, any \PDS\  quiver is a subquiver of some standard \PDS\  quiver, and it induces a natural embedding
\begin{align}\label{emb}
 \eta: \Q_A \to \Q_{m\otimes n}
\end{align}
from the quiver $\Q_A$ of $A$ to a quiver $\Q_{m\otimes n}$ of the form shown in \Pic \ref{fig:Am-tensor-An}. Let
\begin{align}
  \widetilde{B} = &\kk\Q_{m\otimes n}, \nonumber \\
  V = & (\Q_{m\otimes n})_0\backslash \mathrm{Im}(\eta), \nonumber  \\
\text{ and }  W = & \{ w\in \mathrm{Im}(\eta) \mid
        \text{ there is a path } \wp \text{ on } \Q_{m\otimes n} \nonumber \\
      & \text{ such that } \source(\wp) \ne w \ne \target(\wp),
        w \text{ is a vertex on } \wp, \nonumber \\
      & \wp=0 \text{ in } \widetilde{B}/\langle \e_{v}^{\widetilde{B}} \mid v\in V \rangle
        \text{ and } \wp \ne 0 \text{ in } A \} \nonumber
\end{align}
\begin{center}
  ($\e_v^{\widetilde{B}}$ is the idempotent of $\widetilde{B}$ corresponding to the vertex $v$ of the quiver of $\widetilde{B}$).
\end{center}
Then
\begin{center}
$\kk\Q_A \cong \big(\widetilde{B}/\langle \e_{v}^{\widetilde{B}} \mid v\in V \rangle\big)(\frakp_W)$,
\end{center}
where $\frakp_W = \frakp(w\mid w\in W)$.
Assume $X = \{\e^A_{v_i}-\e^A_{w_i}\mid 1\le i\le l\}$. Then we have $A_X^{\g}(\frakS_1) = (A/\langle X\rangle)(\frakS_1)$, where $\frakS_1 = \bigcup\limits_{1\le i\le l}\frakp(v_i)$.
Next, we define the following notations for this proof.
\begin{itemize}
  \item For arbitrary set $S$ of some paths on $\Q_A$, write $\eta(S) = \{\eta(s) \mid s\in S\}$;
  \item $\I_A = \langle r_k \mid k\in K \rangle$ is an admissible ideal of $\kk\Q_A$ such that $A=\kk\Q_A/\I_A$ ($K$ is an index set).
\end{itemize}
Then $\frakp_W \cap \{r_k\mid k\in K\} = \varnothing$, and we obtain the following.
\begin{align}
  & A_X^{\g}(\frakS_1)/\I = (A/\langle X\rangle)(\frakS_1)/\I \nonumber \\
\cong & \frac{\left(
  {\big(\widetilde{B}/\langle \e_{v}^{\widetilde{B}} \mid v\in V \rangle\big)(\frakp_W)}
  \Big/ {\langle \eta(\I_A), \eta(X) \rangle}
  \right)(\eta(\frakS_1))}{\eta(\I)} \nonumber \\
\cong & \frac{\left(
  {(\widetilde{B}/\langle \e_{v}^{\widetilde{B}} \mid v\in V \rangle)}
  \big/ {\langle \eta(\I_A), \eta(X) \rangle}
  \right)(\eta(\frakS_1)\cup\frakp_W)}{\eta(\I)}
  \ \ (\text{by Lemma \ref{lemm:supplement}}) \nonumber \\
\cong & \frac{\left(
  {\widetilde{B}/\eta(\I_A)}\big/{\langle \eta(X)\cup\{\e_{v}^{\widetilde{B}} \mid v\in V\} \rangle}
  \right)(\eta(\frakS_1)\cup\frakp_W)}{\eta(\I)}
  \ \ (\text{by Lemma \ref{lemm:ideal}})
\end{align}
Thus we have $A^{\g}_X(\frakS_1)/\I \cong B^{\g}_{Y}(\frakS_2)/\J$, where:
\begin{itemize}
  \item $B=\widetilde{B}/\eta(\I_A)$ is a standard \PDS\ algebra;
  \item $Y = \eta(X)\cup\{\e_v^{\widetilde{B}}\mid v\in V\}
             = \{\e_{v_i}^{\widetilde{B}} - \e_{w_i}^{\widetilde{B}} \mid 1\le i\le l\} \cup \{\e_v^{\widetilde{B}}\mid v\in V\}$;
  \item $\frakS_2 = \eta(\frakS_1)\cup\frakp_W$;
  \item and $\J = \eta(\I)$.
\end{itemize}

(2) This is a direct corollary of (1) by (\ref{formula:NPDSisquotient}) given in Lemma \ref{lemm:NPDSisquotient}.
\end{proof}

\begin{remark}\rm \label{rmk:NPDSisquotient2}
The existence of the standard \PDS\  algebra $B$ in Lemma \ref{lemm:NPDSisquotient2} is not unique,
and we can always choose the standard \PDS\  algebra $B$ of the form $A_{t\otimes t}$ because the existence of embedding
\[ \eta': \Q_A \to \Q_{t\otimes t} \]
can be given by the composition
\[\xymatrix{
\Q_A \ar[rd]_{\text{(\ref{emb})}} \ar[rr]^{\eta'} & & \Q_{t\otimes t} \\
  & \Q_{m\otimes n} \ar[ru]_{\text{subquiver}} &
} \]
where $t$ is a positive integer $\ge\max\{m,n\}$.
\end{remark}

By Proposition \ref{prop:GDSisquotient}, Lemma \ref{lemm:NPDSisquotient2} and Remark \ref{rmk:NPDSisquotient2}, the following result holds.

\begin{proposition} \label{prop:GDSisquotient2}
Let $A=\kk\Q/\I$ be a \GDS\ algebra with admissible ideal $\I$.
Then there are two algebras $A_n$ and $A_m$ of type $\A$ such that $A$ is isomorphic to a quotient of a \gluing algebra of $A_m\otimes_{\kk} A_n$.
\end{proposition}

\section{The Auslander algebras of string algebras}

The Auslander algebra $A^{\Aus}$ of representation-finite algebra is defined by
\[ A^{\Aus} = \End_A\Big(\bigoplus_{M\in\ind(\modcat A)}M\Big). \]
In this part, we study the Auslander algebra of representation-finite string algebra.

\begin{proposition} \label{prop:Aus-is-GDS}
The Auslander algebra $A^{\Aus}$ of representation-finite string algebra $A$ is a \GDS\ algebra.
\end{proposition}

\begin{proof}
Let $\Gamma(\modcat A) = (\Gamma_0, \Gamma_1, \source, \target)$ be the Auslander-Reiten quiver of $A$ and $(\Q^{\Aus}, \I^{\Aus})$ be the bound quiver of $A^{\Aus}$.
Since the vertices and arrows of $\Gamma(\modcat A)$ correspond to the indecomposable right $A$-modules and irreducible morphisms in $\modcat A$, we have the following isomorphism
\[A^{\Aus} = \kk\Q^{\Aus}/\I^{\Aus} \cong (\kk \Gamma(\modcat A)/\I(\modcat A))^{\op}, \]
where $\I(\modcat A)$ is the ideal of the path algebra $\kk \Gamma(\modcat A)$ given by Auslander-Reiten sequence.
By Theorem \ref{thm:BR1987}, any Auslander-Reiten sequence is of the form
\[L  \mathop{\longrightarrow}\limits^{\left[{^{f_1}_{f_2}}\right]} M_1\oplus M_2
      \mathop{\longrightarrow}\limits^{\left[ g_1\ g_2\right]} N \]
\begin{center}
($L$, $M_1$, $M_2$ and $N$ are indecomposable,

at least one of $M_1$ and $M_2$ is non-zero, and

$f_1$, $f_2$, $g_1$, $g_2$ are either irreducible morphisms or zero).
\end{center}
It provides a subquiver of $\Q^{\Aus}$ which is a diamond.

If $L$ is a middle term of some Auslander-Reiten sequence, then it is one of vertices $L_1$ and $L_2$ of some subquiver
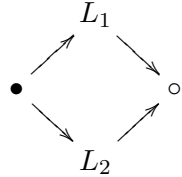
\[\xymatrix@R=0.5cm@C=0.5cm{
& L_1 \ar[rd]        &\\
  \bullet \ar[ru] \ar[rd] && \circ \\
& L_2 \ar[ru]        &
}\]
of $\Gamma(\modcat A)$,
or lies in a subquiver of $\Gamma(\modcat A)$ which is one of the following forms
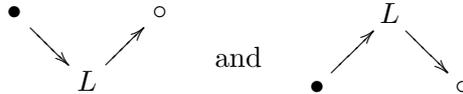
\begin{center}
\begin{tikzpicture}[baseline=-0.25cm]
\draw (0,0) node{$\xymatrix@R=0.5cm@C=0.5cm{
  \bullet \ar[rd] && \circ  \\
& L \ar[ru]        &
}$};
\end{tikzpicture}
 \ \ \text{and} \ \
\begin{tikzpicture}[baseline=-0.25cm]
\draw (0,0) node{$\xymatrix@R=0.5cm@C=0.5cm{
& L \ar[rd]        &\\
  \bullet \ar[ru] && \circ
}$};
\end{tikzpicture}
\end{center}
By the uniqueness of left or right minimal almost split morphism, any point ``$\circ$'' is the vertex in $\Gamma_0$ corresponded by either $M_1$ or $M_2$. We can consider $N$ by the same way.

Thus, the quiver $\Q^{\Aus}$ of $A^{\Aus} = \kk\Q^{\Aus}/\I^{\Aus}$ can be obtained by deleting some vertices of \DS\ quiver,
i.e., there exists a \DS\ algebra $D = \kk\Q_D$ whose quiver is $\Q_D$ without relation such that $D/D\e D \cong \kk\Q^{\Aus}/\J$,
 where $\e$ is an idempotent in $D$ and $\J$ is an admissible ideal generated by some paths on $\Q_D$ which can be seen as paths on $\Q^{\Aus}$ and are zero in $A^{\Aus}$.
Therefore, $A^{\Aus}$ is a \GDS\ algebra because $\I^{\Aus}$ is an admissible ideal whose generator is either a commutative relation or a zero relation and all diamond subquivers of $\Q^{\Aus}$ are commutative in $A^{\Aus}$.
\end{proof}

We obtain our main results.

\begin{theorem} \label{thm:main}
The Auslander algebra of any representation-finite string algebra is a quotient of a \gluing algebra of $\dA_n^e = \dA_n\otimes_{\kk}\dA_n^{\op}$.
\end{theorem}

\begin{proof}
Take $n = \max\{m,n\}$ in Remark \ref{rmk:NPDSisquotient2}. Then, by the isomorphism $\dA_n\cong \dA_n^{\op}$, Proposition \ref{prop:GDSisquotient2} and Proposition \ref{prop:Aus-is-GDS},
there is a set $X=\{\e_{v_i}^{\dA_n^e}-\e_{w_i}^{\dA_n^e} \mid 1\le i\le l\} \cup \{\e_{v_j}^{\dA_n^e} \mid j\in J\}$
and a supplement $\mathfrak{S}$ of $(\dA_n^e)^{\g}_X$ such that the Auslander algebra $A^{\Aus}$ of string algebra $A$ isomorphic to $(\dA_n^e)^{\g}_{X}(\mathfrak{S})/\J$ for some $\J$.
\end{proof}

\section{Examples and applications}

In this part, we provide three examples. In Subsection \ref{subsec:exp1}, we construct an instance for finite-dimensional algebra such that its Auslander algebra is isomorphic to $(\dA_n^e)_X^{\g}(\mathfrak{S})/\langle X\rangle$ with $X\ne \varnothing$ and $\mathfrak{S}\ne\varnothing$.
In Subsection \ref{subsec:exp2} and \ref{subsec:exp3}, we consider two classes of string algebras whose quivers are Dynkin types $\A$ and $\D$,
and further show that their Auslander algebras are isomorphic to quotients of $\dA_n^e$.

\subsection{An example with non-empty \gluing index and non-empty supplement} \label{subsec:exp1}

\begin{example} \rm
Let $A = \kk\Q/\I$ be an algebra with $\Q =$
\[ \xymatrix@R=0.6cm@C=0.5cm{ & 2 \ar[rd]^b & \\ 1 \ar[ru]^a \ar[rr]_c &  & 3 } \]
and $\I = \langle ab\rangle$.
Consider the enveloping algebra $\dA_4^e = \kk\Q_{\dA_4^e}/\I_{\dA_4^e}$ of $\dA_4$,
where $\Q_{\dA_4^e}$ is shown in \Pic \ref{fig:envel-of-A4} (1) and $\I_{\dA_4^e}$ is generated by the commutative relations $a_{i,j}^{\Left}a_{i,j}^{\Upper} - a_{i,j}^{\Down}a_{i,j}^{\Right}$ $(1\le i,j\le 3)$ such that all diamond subquivers of $\Q_{\dA_4^e}$ are commutative.
\begin{figure}[htbp]
\centering
\begin{tikzpicture}
\draw (0,0) node{$\xymatrix@C=0.5cm@R=0.5cm{
 (4,1) \ar@{<-}[d] & (4,2) \ar@{<-}[l] \ar@{<-}[d] & (4,3) \ar@{<-}[l] \ar@{<-}[d] & (4,4) \ar@{<-}[l] \ar@{<-}[d] \\
 (3,1) \ar@{<-}[d] & (3,2) \ar@{<-}[l] \ar@{<-}[d] & (3,3) \ar@{<-}[l] \ar@{<-}[d] & (3,4) \ar@{<-}[l] \ar@{<-}[d] \\
 (2,1) \ar@{<-}[d] & (2,2) \ar@{<-}[l] \ar@{<-}[d] & (2,3) \ar@{<-}[l] \ar@{<-}[d] & (2,4) \ar@{<-}[l] \ar@{<-}[d] \\
 (1,1)             & (1,2) \ar@{<-}[l]             & (1,3) \ar@{<-}[l]             & (1,4) \ar@{<-}[l]
}$};
\draw (0,-2.5) node{(1)};
\end{tikzpicture}
\ \
\begin{tikzpicture}
\draw (0,0) node{$\xymatrix@C=0.5cm@R=0.5cm{
 & (4,2) \ar[r]^{b_4} & (4,3)  & \\
   (3,1) \ar[r]^{a_2} & (3,2) \ar[r]^{a_3} \ar[u]^{b_3}
 & (3,3) \ar[u]^{c_4} \ar@/^1pc/[rdd]^{a_4} & \\
   (2,1) \ar[r]^{c_1} \ar[u]^{a_1} & (2,2) \ar[r]^{c_2} \ar[u]^{b_2}
 & (2,3) \ar[u]^{c_3} & \\
 & & & x \ar@/^1pc/[llu]^{b_1}
}$};
\draw (0,-2.5) node{(2)};
\end{tikzpicture}
\caption{The quivers of $\protect\dA_4^e$ and $A^{\Aus}$.}
\label{fig:envel-of-A4}
\end{figure}
Consider the idempotent $\e = \e_{(4,1)} + \e_{(4,4)} + \e_{(2,4)} + \e_{(1,4)} + \e_{(1,3)} + \e_{(1,1)}$ of $\dA_4^e$,
the quotient $A'=\dA_4^e/\langle\e, \e_{(1,2)}-\e_{(3,4)}\rangle$  is an algebra with the bound quiver $(\Q', \I')$,
where $\Q'$ is shown in \Pic \ref{fig:envel-of-A4} (2) and $\I'$ is generated by $a_1a_2-c_1b_2$, $b_2a_3-c_2c_3$, $b_3b_4-a_3c_4$, $a_2b_3$, $b_1c_2$, $c_3a_4$ and $a_4b_1$.
For any $v\in X:=\{(2,2), (3,2), (3,3), x\}$,
\begin{align}
\frakp(v) = \{ \wp_1\wp_2 \mid &  \wp_1, \wp_2 \text{ are paths on the quiver }
 \Q_{(\dA_4)^{e} } \text{ of } (\dA_4)^{e} \nonumber \\
 & \text{ such that } \source(\wp_1)=v=\target(\wp_2)
   \text{ and } \wp_1\wp_2=0 \text{ in } (\Q',\I') \}, \nonumber
\end{align}
then
\[(\dA_4^e)_X^{\g}\Big(\bigcup_{v\in X}\frakp(v)\Big)
= A'\Big(\bigcup_{v\in X}\frakp(v)\Big) \cong A'' = \kk\Q''/\I'', \]
where $\Q''=\Q'$ and $\I'' = \langle a_1a_2-c_1b_2, b_2a_3-c_2c_3, b_3b_4-a_3c_4 \rangle$.
Therefore,
\begin{align}
  A^{\Aus}
& \cong A''/\langle a_2b_3+\I'', b_1c_2+\I'', c_3a_4+\I'' \rangle \nonumber \\
& \cong (\dA_4\otimes_{\kk}\dA_4^{\op})^{\g}_X\Big(\bigcup_{v\in X}\frakp(v)\Big)
    \Big/ \langle a_2b_3, b_1c_2, c_3a_4 \rangle. \nonumber
\end{align}
\end{example}

\subsection{The Auslander algebra of $\protect\dA_n$} \label{subsec:exp2}
Recall that $\dAA_n$ is the quiver of Dynkin type $\A$ with linearly orientation in this paper.
For any $\dA_n=\kk\dAA_n$, we point out that the Auslander algebra of $\dA_n$ is a quotient of its enveloping algebra in this subsection.
This is an example for string algebra $A$ such that $A^{\Aus}$ is isomorphic to a quotient of the tensor $\dA_n^e$.
To be precise, $X=\varnothing$ and $\mathfrak{S}=\varnothing$ in this case, then $(\dA_n^e)_X^{\g}(\mathfrak{S})/\J$, see Remark \ref{rmk:gluing} (2).
Furthermore, we show that the number of vertices of the quiver $\dAA_n$ is controlled by the representation-type of $\dA_n^{\Aus}$ in this part.

\begin{example} \label{exp:A} \rm
For any $n\ge 1$, we have
\[\dA_n^{\Aus} \cong \dA_n^e/\J_{\A_n} = (\dA_n\otimes_{\kk}\dA_n^{\op})/\J_{\A_n}, \]
where the quiver of $\dA_n^e$ is of the form shown in \Pic \ref{fig:Am-tensor-An} (2) and $\J_{\A_n}=\left\langle\sum\limits_{j=1}^{n-1}\sum\limits_{k=j+1}^{n} \e_{(j,k)}\right\rangle$.
Furthermore, we show that $\dA_n^{\Aus}$ is representation-finite if and only if $n\le 4$.
Indeed, if $n=5$, then the bound quiver of $\dA_5^{\Aus} \cong \dA_5^e/\J_{\A_5}$ has a subquiver which is of the form shown in \Pic \ref{fig:exp:A} (1),
where $a_1b_1-a_2b_2=0$ and $c_1d_1-c_2d_2=0$.
Then we obtain an indecomposable module $M(\lambda\ne 0)$ whose quiver representation is induced by \Pic \ref{fig:exp:A} (2), i.e.,
  \begin{align} \label{formula:Aus-of-D-1}
    \widehat{M(\lambda)} \e^{\dA_5^{\Aus}}_{(i,j)} \cong_{\kk} &
    \begin{cases}
      \kk, & \text{ if } (i,j)\in\{(3,2), (3,3), (4,1), (4,3), (5,1), (5,2)\}; \\
      \kk^2, & \text{ if } (i,j)=(4,2); \\
      0, & \text{ otherwise. }
    \end{cases}
  \end{align}
It is easy to see that $M(\lambda)\cong M(\mu)$ if and only if $\lambda=\mu$. Thus, $\dA_5^{\Aus}$ is representation-infinite.
The proof of this statement is similar to \cite[Chap VII, Proposition 2.5, Lemma 2.6(i) and Corollary 2.7]{ASS2006}.
\begin{figure}[htbp]
\centering
\begin{tikzpicture}
\draw (0,0) node{$\xymatrix@R=0.3cm{
 & 3' \ar[rd]^{b_1} & \\
 1 \ar[ru]^{a_1} \ar[rd]_{a_2} & & 4 \\
 & 3 \ar[rd]^{d_1} \ar[ru]_{b_2}& \\
 2 \ar[ru]^{c_1} \ar[rd]_{c_2} & & 5 \\
 & 3'' \ar[ru]_{d_2} &
}$};
\draw (0,-2.5) node{(1): The quiver $X$, where $a_1b_1-a_2b_2$};
\draw (0,-3) node{$=0$ and $c_1d_1-c_2d_2=0$};
\end{tikzpicture}
\ \
\begin{tikzpicture}
\draw (0,0) node{$\xymatrix@R=0.3cm{
& \kk \ar[rd]^{1} &\\
\kk \ar[ru]^{1} \ar[rd]^{\left[{^1_0}\right]} & & \kk \\
& \kk^2 \ar[ru]^{[1\ 1]} \ar[rd]_{[1\ \lambda]} & \\
\kk \ar[ru]_{\left[{^0_1}\right]} \ar[rd]_{1} & & \kk. \\
& \kk \ar[ru]_{\lambda} &}$};
\draw (0,-2.5) node{(2): an indecomposable};
\draw (0,-3) node{representation of $X$};
\end{tikzpicture}
\caption{The quiver containing subquiver of Euclid type $\widetilde{\mathbb{D}}_4$}
\label{fig:exp:A}
\end{figure}
Furthermore, $\dA_{n}^{\Aus}$ ($n\ge 5$) is representation-infinite because any right $\dA_{5}^{\Aus}$-module $M(\lambda)$ can be naturally seen as a right $\dA_{n}^{\Aus}$-module.
We can check that $\dA_{\le 4}^{\Aus}$ are representation-finite,
to be more precise, we have $\sharp\ind(\modcat \dA_{2}^{\Aus})=5$, $\sharp\ind(\modcat \dA_{3}^{\Aus})=17$ and $\sharp\ind(\modcat \dA_{4}^{\Aus})=56$ by their Auslander-Reiten quivers. Therefore, $\dA_n^{\Aus}$ is representation-finite if and only if $n\le 4$.
\end{example}

\subsection{The Auslander algebra of $\protect\dD_n$} \label{subsec:exp3}
We will describe the representation-type of Auslander algebra $(\dD_n)^{\Aus}$ ($n\ge 4$) by enveloping algebra $\dA_n^e$ in this subsection.

\begin{example} \label{exp:D} \rm
$\dD_n$ is a string algebra, whose Auslander algebra is
\[\dD_n^{\Aus} \cong \dA_n^e/\J_{\D_n} = (\dA_n\otimes_{\kk}\dA_n^{\op})/\J_{\D_n}, \]
where the quiver of $\dA_n^e$ is of the form shown in \Pic \ref{fig:Am-tensor-An} (2) and \[\J_{\D_n} = \left\langle
\left(\sum\limits_{k=1}^{n-2}\sum\limits_{j=k+1}^n \e_{(k,j)} - \e_{(n-2,n-1)} \right)
+ \sum\limits_{{1\le j\le n} \atop {j\ne n-1}}\e_{(n,j)}
\right\rangle. \]
Next, we show that $\dD_n^{\Aus}$ is representation-finite if and only if $n\le 5$.
First of all, $\dD_6^{\Aus}$ is representation-infinite because any indecomposable module $M(\lambda\ne0)$ (its quiver representation is shown in \Pic \ref{fig:exp:A} (2)) induces a right $\dD_6^{\Aus}$-module,
and $M(\lambda)\cong M(\mu)$ if and only if $\lambda=\mu$. Thus, $\dD_{6}^{\Aus}$ is representation-infinite.
The representation-infiniteness of $\dD_{\ge 6}^{\Aus}$ can be proved by similar way.
We can check that $\dD_{\le 5}^{\Aus}$ is representation-finite, to be more precise, we have
$\sharp\ind(\modcat\dD_4^{\Aus})=40$ and $\sharp\ind(\modcat\dD_5^{\Aus}) = 109$ by Auslander-Reiten quiver.
Therefore, $\dD_n^{\Aus}$ is representation-finite if and only if $n\le 5$.
\end{example}


\vspace{10pt}
\noindent
\textbf{Acknowledgements.}
Jian He was supported by the National Natural Science Foundation of China (Grant No. 12171230) and Youth Science and Technology Found of Gansu Provincial (Grant No. 23JRRA825).
Yu-Zhe Liu was supported by National Natural Science Foundation of China (Grant Nos. 12401042, 12171207), Guizhou Provincial Basic Research Program (Natural Science) (Grant No. ZK[2024]YiBan066), and Scientic Research Foundation of Guizhou University (Grant Nos. [2022]53, [2022]65, [2023]16).



 \def\cprime{$'$}
\providecommand{\bysame}{\leavevmode\hbox to3em{\hrulefill}\thinspace}
\providecommand{\MR}{\relax\ifhmode\unskip\space\fi MR }
\providecommand{\MRhref}[2]{%
  \href{http://www.ams.org/mathscinet-getitem?mr=#1}{#2}
}
\providecommand{\href}[2]{#2}

\end{document}